\documentclass[12pt,oneside,reqno]{amsart}
\usepackage{amsmath,amssymb}
\usepackage{amscd}
\usepackage{tikz}
\theoremstyle{plain}
\newtheorem{thm}{Theorem}[section]
\newtheorem{theorem}[thm]{Theorem}

\newtheorem{lemma}[thm]{Lemma}

\newtheorem{proposition}[thm]{Proposition}
\theoremstyle{definition}

\newtheorem*{remark}{Remark}

\newtheorem*{claim}{Claim}

\newtheorem{definition}[thm]{Definition}

\numberwithin{equation}{section}
\newcommand{\0}{{\mathcal O}}

\newcommand{\BP}{{\mathbb P}}

\newcommand{\Q}{{\mathbb Q}}
\newcommand{\R}{{\mathbb R}}

\newcommand{\Z}{{\mathbb Z}}

\newcommand{\Aut}{{\rm Aut}}

\newcommand{\id}{{\rm id}}

\newcommand{\disc}{{\rm disc}}

\def\Hom{\mathop{\rm Hom}\nolimits}

\def\Mat{\mathop{\rm Mat}\nolimits}

\title[K3 surfaces]{Automorphisms of K3 surfaces with Picard number two}
\author{Kwangwoo Lee}
\begin{document}
\maketitle  \setcounter{tocdepth}{1}

\begin{abstract}
It is known that the automorphism group of a K3 surface with Picard number two is either an infinite cyclic group or an infinite dihedral group when it is infinite. In this paper, we study the generators of such automorphism groups. We use the eigenvector corresponding to the spectral radius of an automorphism of infinite order to determine the generators.
\end{abstract}

\section{Introduction}\label{Introduction}
The aim of this paper is to give some conditions on the generators of the automorphism group of a K3 surface of Picard number $2$ (Theorem \ref{theorem1} and Theorem \ref{theorem2}). For a K3 surface $X$ of rank two Picard lattice, Galluzzi, Lombardo and Peters \cite{GLP} applied the classical theory of binary quadratic forms to prove that the automorphism group $\Aut(X)$ is trivial or $\Z_2$ if it is finite. Moreover, if it is infinite, the automorphism group is an infinite cyclic group or an infinite dihedral group.
In this paper, we find conditions for the generators of the automorphism group by using the eigenvector corresponding to the spectral radius of an automorphism of infinite order.

Let $g$ be an automorphism of a compact complex surface $X$. It is known that the topological entropy $h(g)$ is determined by the spectral radius $\rho$ of $g^*$ acting on $H^*(X)$, that is, $h(g)=\log \rho(g^*\vert H^2(X)).$ If $h(g)>0$, then a minimal model for $X$ is either a K3 surface, an Enriques surface, a complex torus or a rational surface \cite{C}. 
In the sense of dynamics of automorphisms, it is a natural question to find a minimal possible entropy.
For example, for a K3 surface, one constructs an automorphism synthetically by lattice theory and Torelli theorem to find the minimal entropy (\cite{Mc}).
However, in this paper, we go the other way around. That is, using topological entropies, we determine automorphisms of $X$. More precisely, by finding eigenvectors of an automorphism, we can find the generators of the automorphism group of a K3 surface. 

The main observation of this idea is the fact that the solutions of the Pell equation associated with a non-square number form an infinite group generated by \textit{Pell multiples} of finite solutions (cf. Section \ref{Pell equations}). For some $k\in \Z\setminus \{0,-1\}$, if we consider a non-empty set $A_k$ of divisors each of which has the self-intersection number $2k$, then by $g\in \Aut(X)$, $g^*(D)\in A_k$ for any $D\in A_k$, where $g^*:=g^*\vert S_X$ and $S_X$ is the Picard lattice of $X$. In particular, for $(u,v)\in A_k$, if we consider a sequence of $(u_n,v_n):=g^{n*}(u,v)$ and if the ratio $u_n/v_n$ converges to $U/V$ as $n$ increases, then $(U,V)$ will be the eigenvector corresponding to the spectral radius of $g^*\vert S_X$.
We use this eigenvector to find $g$. Furthermore, using this eigenvector and $g$, we can also determine anti-symplectic involutions when $\Aut(X)\cong\Z_2\ast \Z_2$.

Let $X$ be a K3 surface with Picard lattice $S_X$ whose self-intersection matrix is
\begin{equation}\label{Gram matrix}
Q_{S_X}=\begin{pmatrix} 2a&b\\b&2c \end{pmatrix}
\end{equation}
for some basis with $d:=-\disc(S_X)=b^2-4ac>0$. Let $g$ be an automorphism of $X$ with $g^*\vert S_X$ given by the matrix 
\begin{equation}\label{Automorphism}
\begin{pmatrix} \alpha&\beta\\ \gamma&\delta \end{pmatrix}. 
\end{equation}

\begin{theorem}\label{theorem1}
For a K3 surface $X$ with intersection matrix of Picard lattice $S_X$ as above, an automorphism $g$ of infinite order acting on $S_X$ as in \eqref{Automorphism} satisfies 
\begin{equation}\label{Relations1}
\gamma=-\frac{a}{c}\beta \text{  , } \delta=\alpha-\frac{b}{c}\beta \text{  and  } \alpha^2-\frac{b}{c}\alpha\beta+\frac{a}{c}\beta^2=1.
\end{equation}
Moreover, $g^*|S_X$ is a power of an isometry $h$ of $S_X$ defined by the matrix
\begin{equation}\label{isometry}
h=\begin{pmatrix} \alpha_1&\beta_1\\ -\frac{a}{c}\beta_1&\alpha_1-\frac{b}{c}\beta_1 \end{pmatrix},
\end{equation}
where $(2\alpha_1-b\frac{\beta_1}{c},\frac{\beta_1}{c})$ is the minimal positive solution of Pell equation $x^2-dy^2=4$.
\end{theorem} 

Furthermore, suppose that we have $\Aut(X)\cong \Z_2\ast \Z_2$. Then we have the following result.

\begin{theorem}\label{theorem2}
For a K3 surface $X$ with Picard lattice $S_X$ as above, an involution $\iota$ acting on $S_X$ by the matrix in \eqref{Automorphism} satisfies
\begin{equation}\label{Relations2}
\delta=-\alpha, \gamma=-\frac{b}{c}\alpha+\frac{a}{c}\beta \text{  and  } \alpha^2-\frac{b}{c}\alpha\beta+\frac{a}{c}\beta^2=1.
\end{equation}
\end{theorem} 

\begin{remark}\label{anti-symplectic only}
It is known that a symplectic involution occurs only if Picard number is greater than $8$ (cf. \cite{N1}). 
\end{remark}
 

\begin{remark}
By Lemma \ref{criterion} in Section \ref{Lattices}, we can easily see whether an isometry acts on $S_X$ as $-\id$ or not. If it acts on $S_X$ as $-\id$, then it extends to an isometry of $H^2(X,\Z)$. Moreover, if it preserves the ample cone, then by Torelli theorem, it extends to an anti-symplectic involution.
\end{remark}
The structure of this paper is the following: in Section \ref{Preliminaries} we recall some results about Pell equations and lattices. In Section \ref{Proof} we prove Theorem \ref{theorem1} and Theorem \ref{theorem2}.  In Section \ref{Applications} we apply our results to several examples  to find the generators of the automorphism group of a K3 surface of Picard number $2$.

\section{Preliminaries} \label{Preliminaries}
\subsection{Pell equations}\label{Pell equations}
For a positive integer $d$, an equation of the form
\begin{equation}\label{Pell equation}
u^2-dv^2=1
\end{equation}
is called a \textit{Pell equation}. We are interested in solutions $(u,v)$ where $u$ and $v$ are integers. Solutions with $u>0$ and $v>0$ will be called \textit{positive solutions}. It is known in \cite{L} that for every non-square positive integer $d$, the equation \eqref{Pell equation} has a nontrivial solution with $v\neq 0$. Moreover, the solutions of Pell equation can be generated from the smallest positive solution $(u_1,v_1)$ of \eqref{Pell equation}.

\begin{theorem}\label{solutions of Pell equation} (\cite{AG}, Sec. 6.6. Theorem 7)
If $d$ is a square, the only solutions of \eqref{Pell equation} are $u=\pm 1$ and $v=0$.

If $d$ is not a square, let $(u_1,v_1)$ be the smallest positive solution of \eqref{Pell equation} and write $\alpha=u_1+v_1\sqrt{d}$, then all solutions of \eqref{Pell equation} are
\begin{equation*}
\{(\pm u_n,v_n)\mid u_n, v_n\in \Z \text{ such that }  u_n+v_n\sqrt{d}=(\alpha)^n, n\in \Z\}.
\end{equation*}
\end{theorem}
\begin{remark} All solutions of \eqref{Pell equation} are units of $\Z[\sqrt{d}]$.
\end{remark} 

More generally, for $m\in \Z\setminus \{0\}$, the equation
\begin{equation}\label{Generalized Pell equation}
u^2-dv^2=m
\end{equation} 
is called a \textit{generalized Pell equation}.

\begin{theorem}\label{generalized Pell equation} (\cite{AG}, Sec. 6.6. Theorem 9)
If the generalized Pell equation \eqref{Generalized Pell equation} has one solution, then it has infinitely many solutions.
\end{theorem}
\begin{proof} 
If $(a,b)$ is a solution of \eqref{Generalized Pell equation}, then for any solution $(u_n,v_n)$ of \eqref{Pell equation}, $(u'_n, v'_n)$ defined by $u'_n+v'_n\sqrt{d}=(u_n+v_n\sqrt{d})(a+b\sqrt{d})$ is also a solution of \eqref{Generalized Pell equation}.
\end{proof}
By a \textit{Pell multiple}, we mean the coefficients $(au+dbv, av+bu)$ of the product $(a+b\sqrt{d})(u+v\sqrt{d})$.

\begin{theorem}(\cite{AA}, Theorem 4.1.3) \label{orbits of general Pell equation}
For a general Pell equation $u^2-dv^2=m$ with $m\neq 0$, it is known that there is a finite set of solutions such that every solution is a Pell multiple of one of these solutions. In other words, there are finite number of orbits of solutions under the group of units $(\Z[\sqrt{d}])^*$.
\end{theorem}

In particular for $m=4$, we have the following result.
\begin{theorem}(\cite{AA}, Theorem 4.4.1) \label{4-Pell equation}
Let $d$ be a non-square positive integer. If $(u_1,v_1)$ is the smallest positive solution of $u^2-dv^2=4$, then all solutions are generated by powers of $\alpha=\frac{u_1+v_1\sqrt{d}}{2}$ in the sense that writing $\alpha^n=\frac{u_n+v_n\sqrt{d}}{2}$, $(\pm u_n,v_n)$ is a new solution and all solutions can be obtained in that way.
\end{theorem}
\begin{remark} If $d=b^2-4ac>0$, then $d\equiv 0 \text{ or } 1 \mod 4$. When $d\equiv 1 \mod 4$, for any solution to $u^2-dv^2=4$, $u$ and $v$ are both odd or both even. If the smallest positive solution has both $u_1$ and $v_1$ even, then all solutions have both $u_n$ and $v_n$ even. In this case, every solution to \eqref{Pell equation} is just one-half of a solution to $u^2-dv^2=4$. If $u_1$ and $v_1$ are both odd, then $d\equiv 5 \mod 8$, every third solution has $u$ and $v$ even, and all other solutions have $u$ and $v$ odd. In this case, every solution of \eqref{Pell equation} is just one-half of one of the solutions of $u^2-dv^2=4$ that has both $u$ and $v$ even. 

When $d\equiv 0 \mod 4$, any solution of $u^2-dv^2=4$ has $u$ even. If the smallest positive solution has $v$ even, then all solutions have $v$ even and every solution of \eqref{Pell equation} is just one-half of one of the solutions of $u^2-dv^2=4$. If the smallest solution has $v$ odd, then every other solution has $v$ even and every other solution has $v$ odd. In this case, every solution of \eqref{Pell equation} is just one-half of one of the solutions of $u^2-dv^2=4$ that has $u$ and $v$ both even.
\end{remark} 

\subsection{Lattices}\label{Lattices}
A \textit{lattice} is a pair $(L,b)$ of a free finite rank $\Z$-module $L$ together with a bilinear form $b:L\times L\rightarrow \Z$. A lattice is \textit{even} if $b(x,x)\in 2\Z$ for any $x\in L$, \textit{odd} otherwise. The \textit{discriminant} $\disc(L)$ is the determinant of the matrix of the bilinear form. A lattice is called \textit{non-degenerate} if the discriminant is non-zero and \textit{unimodular} if the discriminant is $\pm1$. If the lattice $L$ is non-degenerate, the pair $(s_+,s_-)$, where $s_{\pm}$ denotes the multiplicity of the eigenvalue $\pm 1$ for the quadratic form associated to $L\otimes \R$, is called a \textit{signature} of $L$. An \textit{isometry} of a lattice is an isomorphism preserving the bilinear form. The \textit{orthogonal group} $O(L)$ consists of all isometries of $L$.

 For a lattice $(L,b)$, the dual lattice $L^*$ is defined by
\begin{equation*} 
L^*=\Hom_{\Z}(L,\Z)=\{x\in L\otimes \Q \mid b(x,y)\in \Z \text{ for any } y\in L\}.
\end{equation*}
We have a natural inclusion $L\hookrightarrow L^*$ and the \textit{discriminant group} of $L$ is $A(L)=L^*/L$. 
The bilinear form on $L$ induces a symmetric bilinear form $b^*:L^*\times L^*\rightarrow \Q$. Moreover, $b^*$ induces a symmetric bilinear form $b_L:A(L)\times A(L)\rightarrow \Q/\Z$ and thus a quadratic form $q_L:A(L)\rightarrow \Q/\Z$. 

Whenever $L$ is even, $q_L$ takes values in $\Q/2\Z$. By $O(A(L))$ we denote the group of automorphisms of $A(L)$ preserving $q_L$. The inclusion of $L$ into $L^*$ yields a homomorphism $\Phi:O(L)\rightarrow O(A(L))$. For a non-degenerate lattice $L$ of signature $(1,k)$ with $k>0$, we have the decomposition
\begin{equation}
\{x\in L\otimes \R \mid x^2>0 \}=C_L\cup (-C_L)
\end{equation}
into two disjoint cones. We define
\begin{equation}
O^+(L):=\{ g \in O(L) \mid g(C_L)=C_L \}.
\end{equation}
Note that $O^+(L)$ is a subgroup of $O(L)$ of index $2$.

We state some results about lattices which will be used in later sections.

\begin{theorem} \label{Existence} (\cite{N2}, Theorem 1.14.4.) For any even lattice $L$ of signature $(1,\rho)$ with $\rho\leq 9$, there exists a projective K3 surface $X$ such that $S_X\cong L$.
\end{theorem}

An embedding $S\hookrightarrow L$ of lattices is called \textit{primitive} if $L/S$ is free.
\begin{proposition} \label{Glueing of lattices} (\cite{N2}, Proposition 1.6.1.) A primitive embedding of an even lattice $S$ into an even unimodular lattice $L$, in which the orthogonal complement of $S$ is isomorphic to $K$, is determined by an isomorphism $\gamma:A(S)\xrightarrow{\sim} A(K)$ for which $q_K\circ \gamma=-q_S$.
\end{proposition}

\begin{lemma} \label{criterion} Let $L$ be a non-degenerate even lattice of rank $n$. For $g\in O(L)$ and $\epsilon\in \{\pm 1\}$, $g$ acts on $A(L)$ as $\epsilon\cdot \id$ if and only if $(g-\epsilon\cdot I_n)\cdot Q_L^{-1}$ is an integer matrix, where $Q_L$ is the intersection matrix of $L$.
\end{lemma}
\begin{proof}  Since $L^*$ is generated by the columns $\{v_i\}_{i=0}^n$ of $Q_L^{-1}$, $g\mid_{L^*/L}=\pm \id\mid_{L^*/L}$ if and only if $(g\mp \id)(v_i)$ are in $L$ for all $i$. This is equivalent to  $(g-\epsilon\cdot I_n)\cdot Q_L^{-1}$ is an integer matrix.
\end{proof}

\section{Proof} \label{Proof}
Let $L$ be an even lattice of signature $(1,1)$ with intersection matrix given by \eqref{Gram matrix}
with $d:=-\disc(L)=b^2-4ac>0$. By Theorem \ref{Existence}, there is a K3 surface $X$ whose Picard lattice $S_X\cong L$. 
 It is known that $d$ is a square number if and only if there is a $D\in S_X$ with $D^2=0$. Suppose that $d$ is not a square number. Moreover, we assume that $S_X$ has no $D$ whose self-intersection $D^2=-2$. Hence we have that $\Aut(X)$ is infinite (cf. Theorem \ref{Infinite automorphism group}). 

We consider $A_k=\{D=(x,y)\in S_X\mid D^2=2ax^2+2bxy+2cy^2=2k\}$ for some $k\in \Z$ such that $A_k\neq \emptyset$.
Then $A_k$ consists of $(A_k)_{\pm}=\{(x,y)\mid x=\frac{-by\pm z}{2a} \text{ such that } z^2-dy^2=4ak\}$.
For $(x_0,y_0)\in (A_k)_{\pm}$ with $x_0=\frac{-by_0\pm z_0}{2a}$ and $z_0^2-dy_0^2=4ak$, we have that $(x_n,y_n)\in (A_k)_{\pm}$ with $x_n=\frac{-by_n\pm z_n}{2a}$ and $z_n^2-dy_n^2=4ak$ by the proof of Theorem \ref{generalized Pell equation}. Here $(z_n,y_n)$ is obtained by $z_n+y_n\sqrt{d}=(u_n+v_n\sqrt{d})(z_0+y_0\sqrt{d})$ for some solution $(u_n,v_n)$ of Pell equation $u^2-dv^2=1$. In other words, $(z_n,y_n)$ is a Pell multiple of $(z_0,y_0)$. 


\begin{lemma}\label{Limit}
For these $(x_n,y_n)\in (A_k)_{\pm}$ (resp.), $\frac{x_n}{y_n}$ converges to $\frac{-b\pm \sqrt{d}}{2a}$ (resp.) as $n$ increases.
\end{lemma}
\begin{proof} By the proof of Theorem \ref{generalized Pell equation}, we have $(x_n,y_n)$ with $x_n=\frac{-by_n+z_n}{2a}$, where $(z_n,y_n)$ is obtained by $z_n+y_n\sqrt{d}=(u_n+v_n\sqrt{d})(z_0+y_0\sqrt{d})$ with $u_n^2-dv_n^2=1$. Now $\frac{z_n}{y_n}=\frac{u_nz_0+dv_ny_0}{v_nz_0+u_ny_0}=\frac{z_0\frac{u_n}{v_n}+dy_0}{z_0+\frac{u_n}{v_n}y_0}$ converges to $\sqrt{d}$ as $\frac{u_n}{v_n}$ converges to $\sqrt{d}$ as $n$ increases. Now for $(x_n,y_n)\in (A_k)_{+}$, $\frac{x_n}{y_n}=\frac{-b+z_n/y_n}{2a}$ converges to $\frac{-b+\sqrt{d}}{2a}$ as $n$ increases. 

Moreover, by Theorem \ref{orbits of general Pell equation}, we have only a finite set of such $(z_0,y_0)$ and for each of which we have the same result. Similarly this also holds for $(x_n,y_n)\in (A_k)_{-}$.
\end{proof}

\subsection{Proof of Theorem \ref{theorem1}}\label{Proof of Theorem 1}
Let $g$ be an automorphism of infinite order of $X$ whose action on $S_X$ is given by \eqref{Automorphism}.
Then $A_k\neq \emptyset$ for some $k\neq 0,-2$ (cf. Theorem \ref{Infinite automorphism group}). For $(x_0,y_0)\in A_k$, let $(x_n,y_n)=g^{n*}(x_0,y_0)$, where $g^{n*}$ is the induce isometry on $S_X$. Hence we have that $(x_n,y_n)\in A_k$ and by Lemma \ref{Limit}, $\frac{x_n}{y_n}$ converges to $\frac{-b+\sqrt{d}}{2a}$. Recall that $(x_n,y_n)\in \Z^2$.

For $(x_n,y_n)=g^{n*}(x_0,y_0)$, $\frac{x_n}{y_n}$ converges to $\frac{x}{y}=\frac{-b+\sqrt{d}}{2a}$ as $n$ increases, hence the ratio $\frac{x}{y}$ indicates the direction of eigenvector of $g^*$. Hence for $(x,y)$ such that $\frac{x_n}{y_n}$ converges to $\frac{x}{y}$, $g*$ preserves the ratio, i.e., for 
\begin{equation}\label{g action}
 \begin{pmatrix} \alpha&\beta\\ \gamma&\delta \end{pmatrix}\begin{pmatrix} x\\ y \end{pmatrix}=\begin{pmatrix} \alpha x+\beta y\\ \gamma x+\delta y \end{pmatrix},
\end{equation}
we have that 
\begin{equation}\label{Ratio}
\frac{\alpha x+\beta y}{\gamma x+\delta y}=\frac{x}{y}=\frac{-b+\sqrt{d}}{2a} \hspace{3mm} \text{ or } \hspace{3mm}\frac{\alpha\frac{-b+\sqrt{d}}{2a}+\beta }{\gamma\frac{-b+\sqrt{d}}{2a}+\delta }=\frac{-b+\sqrt{d}}{2a}.
\end{equation}
Then we have that 
\begin{equation}
\alpha(-b+\sqrt{d})+2a\beta=\frac{\gamma(2b^2-4ac-2b\sqrt{d})}{2a}+\delta(-b+\sqrt{d}).
\end{equation}
Now since this is an element of $\Q[\sqrt{d}]$, we have that 
\begin{equation}\label{two conditions}
\alpha+\frac{b}{a}\gamma-\delta=0 \hspace{3mm} \text{ and }\hspace{3mm} a\beta+c\gamma=0.
\end{equation}  
This gives the first two conditions in \eqref{Relations1}.
Since we assumed that $d$ is not a square number, $c\neq 0$. 
Moreover, since $g^*$ is an isometry of $S_X$, we have that $g^{*tr}Q_{S_X}g^*=Q_{S_X}$, where $Q_{S_X}$ is the intersection matrix of $S_X$ as in \eqref{Gram matrix}. This and condtions in \eqref{two conditions} give the last identity in \eqref{Relations1}. This proves the first part of the theorem.
 
Next, we will see that the generator of infinite order is a power of such a minimal isometry in the sense that the minimal isometry is associated with the minimal positive solution of some Pell equation. 

From the last condition in \eqref{Relations1}, we have that $\alpha=\frac{b\frac{\beta}{c}+z}{2}$ with $(z, \frac{\beta}{c})$ is a solution of the Pell equation 
\begin{equation}\label{d'-Pell equation}
u^2-dv^2=4.
\end{equation}
Let $(z_1,\frac{\beta_1}{c})$ be the minimal positive solution of \eqref{d'-Pell equation}.
Then by Theorem \ref{4-Pell equation}, any solution $(z_k,\frac{\beta_k}{c})$ is given by a power of $\frac{z_1+\frac{\beta_1}{c}\sqrt{d}}{2}$, that is, $\frac{z_k+\frac{\beta_k}{c}\sqrt{d}}{2}=(\frac{z_1+\frac{\beta_1}{c}\sqrt{d}}{2})^k$ with $k\in \Z$. For $\alpha_1=\frac{b\frac{\beta_1}{c}+z_1}{2}$, let
\begin{equation}\label{h1}
h=\begin{pmatrix} \alpha_1&\beta_1 \\-\frac{a}{c}\beta_1 &\alpha_1-\frac{b}{c}\beta_1 \end{pmatrix}.
\end{equation}

\begin{claim}\label{claim}
Any matrix in \eqref{Automorphism} satisfying the relations \eqref{Relations1} is a power of $h$.
\end{claim}

\begin{proof}
For the induction argument, let $(\alpha_{k},\beta_k)$ be the first row of $h^k$, i.e., 
\begin{equation}\label{h^k}
h^k=\begin{pmatrix} \alpha_k&\beta_k \\-\frac{a}{c}\beta_k &\alpha_k-\frac{b}{c}\beta_k \end{pmatrix}.
\end{equation}
Then 
\begin{equation}\label{power of h}
h^{k+1}=\begin{pmatrix} \alpha_1\alpha_k-\frac{a}{c}\beta_1\beta_k&\alpha_1\beta_k+\alpha_k\beta_1-\frac{b}{c}\beta_1\beta_k \\ \frac{a}{c}(-\alpha_k\beta_1-\alpha_1\beta_k+\frac{b}{c}\beta_1\beta_k) & \alpha_1\alpha_k-\frac{a}{c}\beta_1\beta_k+\frac{b}{c}(-\alpha_1\beta_k-\alpha_k\beta_1+\frac{b}{c}\beta_1\beta_k) \end{pmatrix}.
\end{equation} 
It can be easily shown that $\alpha_{k+1}=\alpha_1\alpha_k-\frac{a}{c}\beta_1\beta_k$ and $\beta_{k+1}=\alpha_1\beta_k+\alpha_k\beta_1-\frac{b}{c}\beta_1\beta_k$ and hence $h^{k+1}$ satisfies \eqref{Relations1}. Hence, the automorphism $g$ of infinite order of $X$ is some power of $h$ preserving the ample cone (or positive cone) of $X$ by Torelli theorem. This proves Theorem \ref{theorem1}.
\end{proof}

 

\begin{remark} Note that for $\alpha_k=\frac{\frac{b}{c}\beta_k-z_k}{2}$, the corresponding isometry will reflect the positive cone and the negative cone.
\end{remark}

\subsection{Proof of Theorem \ref{theorem2}}\label{Proof of Theorem 2}
In this section, we find some conditions on the involutions if we have $\Aut(X)\cong \Z_2\ast \Z_2$. 

\begin{definition}(\cite{GLP}, Section 3.2.) \label{Ambiguous}
A lattice $L$ is \textit{ambiguous} if $L$ admits an isometry $P$ with $\det P=-1$.
\end{definition}

\begin{theorem}(\cite{GLP}, Corollary 1.)\label{Infinite automorphism group}
For $X$ a K3 surface with Picard number $2$ the group $\Aut(X)$ is finite precisely when the Picard lattice $S_X$ contains divisors $L$ with $L^2=0$ or with $L^2=-2$. If $S_X$ does not contain such divisors and if moreover $S_X$ is not ambiguous, then $\Aut(X)$ is infinite cyclic, but if $S_X$ is ambiguous, then $\Aut(X)$ is either infinite cyclic or the infinite dihedral group.
\end{theorem}

\begin{proof}[Proof of Theorem \ref{theorem2}]
Suppose that $\sigma$ and $\tau$ are the generators of $\Z_2\ast \Z_2$. By Remark \ref{anti-symplectic only}, these are anti-symplectic involutions. Hence if we let $g=\sigma\circ \tau$, then $g$ is symplectic of infinite order. Let $\iota$ be an involution of $X$. Then $\iota=g^n\circ \sigma$ or $\tau\circ g^n$ for some $n$.
Let $\{v,w\}$ be the eigenvectors of $g^{n*}| S_X\otimes \R$ corresponding to eigenvalues $\rho$ and $1/\rho$ such that $g^{n*}v=\rho v$ and $g^{n*}w=\frac{1}{\rho}w$, where $\rho$ is the spectral radius of $g^{n*}|S_X$. 

Now since $g=\sigma\circ \tau$, for either $\iota=g^n\circ \sigma$ or $\tau\circ g^n$,
$\iota(v)=rw$ and $\iota(w)=r^{-1} v$ for some $r\in \R\setminus \{0\}$. 
If we let 
\begin{equation}\label{tau}
\iota=\begin{pmatrix}\alpha&\beta \\ \gamma&\delta\end{pmatrix},
\end{equation}
then as in \eqref{Ratio} this implies that for $v=(x,y)$ with $\frac{x}{y}=\frac{-b+\sqrt{d}}{2a}$, the ratio $\frac{x}{y}$ is interchanged with its conjugate  via $\iota$. 
Hence, we have that
\begin{equation}\label{Ratio of tau}
\frac{\alpha\frac{-b\pm \sqrt{d}}{2a}+\beta}{\gamma\frac{-b\pm \sqrt{d}}{2a}+\delta}=\frac{-b\mp \sqrt{d}}{2a}.
\end{equation}
 This implies that 
\begin{equation}
\frac{-b\pm \sqrt{d}}{2a}\alpha+\beta=\frac{c}{a}\gamma+\frac{-b\mp \sqrt{d}}{2a}\delta.
\end{equation}
Thus we have that 
\begin{equation}\label{Involution condition 1}
\delta=-\alpha, \hspace{2mm}\gamma=-\frac{b}{c}\alpha+\frac{a}{c}\beta.
\end{equation}
Moreover, since $\det(\iota)=-1$, we also have that
\begin{equation}\label{Involution condition 2}
\alpha^2-\frac{b}{c}\alpha\beta+\frac{a}{c}\beta^2=1.
\end{equation}
\end{proof}


\begin{remark}\label{a=c}
Consider the Picard lattice \eqref{Gram matrix} with $a=c$. In this case, an involution $\tau$ can be obtained from a power of $h$ by interchanging columns. In other words, if we let $E_{12}\in \Mat_{2\times 2}(\Z)$ be the matrix interchanging two columns, then $\tau:=\tau_m=h^mE_{12}$ for some $m$. Similarly $\sigma:=\sigma_l=h^lE_{12}$ for some $l$. Hence we have that $\sigma\circ \tau=g=h^k=h^lE_{12}h^mE_{12}=h^{l-m}$, i.e., $k=l-m$.
\end{remark}


\section{Applications} \label{Applications}
In this section, we apply our results to several examples.
\subsection{Example 1} In \cite{M}, Mori showed that there is a non-singular quartic surface $X$ in $\BP^3$ with a non-singular curve $C$ of degree $d$ and genus $g$ if and only if (1) $g=d^2/8+1$, or (2) $g<d^2/8$ and $(d,g)\neq (5,3)$. More generally, we refer \cite{K}.

Let $X$ be a quartic hypersurface in $\BP^3$ whose Picard lattice $S_X$ has the intersection matrix
\begin{equation}\label{genus 2}
\begin{pmatrix} 4& d\\ d&2g-2 \end{pmatrix}
\end{equation}
generated by $\{H=\mathcal{O}_X(1),C\}$. 
If $g=d^2/8+1$, then the discriminant of \eqref{genus 2} is zero. In this case, by Theorem \ref{Infinite automorphism group}, the automorphism group is finite. Hence we assume that $g<d^2/8$.

We consider a K3 surface $X$ whose Picard lattice $S_X$ has the following intersection matrix
\begin{equation}\label{Example 1}
Q_{S_X}=\begin{pmatrix} 4&2n\\2n&4 \end{pmatrix}
\end{equation} 
with $d:=-\disc(S_X)=4(n^2-4)>0$. Since $X$ has no divisors of self-intersection number $0$ or $-2$, $\Aut(X)$ is either $\Z$ or $\Z_2\ast  \Z_2$ by Theorem \ref{Infinite automorphism group}.

First, we determine the generator of automorphisms of infinite order in $\Aut(X)$.
By the relations \eqref{Relations1}, an automorphism $g$ of infinite order with $g^*\vert S_X$ in \eqref{Automorphism} satisfies $\gamma=-\beta, \delta=\alpha-n\beta$ and $\alpha^2-n\alpha\beta+\beta^2=1$. Now by Theorem \ref{theorem1}, $g^*\vert S_X$ is a power of $h$, where 
\begin{equation}\label{h}
h=\begin{pmatrix} n&1\\-1&0 \end{pmatrix}.
\end{equation} 
Indeed, we have the following.



\begin{proposition}\label{Example1} 
For  a K3 surface $X$ whose intersection matrix of Picard lattice $S_X$ is given in \eqref{Example 1} with $d=-\disc(S_X)=4(n^2-4)>0$, $\Aut(X)\cong \Z$. The generator is as follows:
\begin{enumerate}
\item if $n$ is even, $g^*=h^4$ is the generator of $\Z$ and symplectic.
\item if $n$ is odd ($\neq 3$), $g^*=h^6$ is the generator of $\Z$ and symplectic.
\item if $n=3$, $g^*=h^3$ is the generator of $\Z$ and anti-symplectic. 
\end{enumerate}
\end{proposition}

\begin{proof}
As Claim in Section \ref{Proof of Theorem 1}, let $\alpha_k, \beta_k$ be the first row of $h^k$. Then we have $\alpha_{k+1}=n\alpha_k-\beta_k$ and $\beta_{k+1}=\alpha_k$. For example, $(\alpha_1,\beta_1)=(n,1)$, $(\alpha_0,\beta_0)=(1,0)$ and $(\alpha_{-1},\beta_{-1})=(0,-1)$, etc. 
It is easily shown that $(h^4-I_2)Q_{S_X}^{-1}$ is an integer matrix for even $n$ and $(h^6-I_2)Q_{S_X}^{-1}$ is an integer matrix for odd $n$. Moreover, these powers are the minimal in order to be an integer matrix. In other words, for even number $n$, $(h^k\pm I_2)Q_{S_X}^{-1}$ is not an integer matrix for any $k\leq 3$. Similarly, for $n\neq 3$ odd,  $(h^k\pm I_2)Q_{S_X}^{-1}$ is not an integer matrix for any $k\leq 5$. 

Now by Proposition \ref{Glueing of lattices}, $h^4$ ($n$ even) or $h^6$ ($n\neq 3$ odd) can be extended to an isometry of $H^2(X,\Z)$ and by Torelli theorem it defines an automorphism of $X$. 
Hence $h^4$ for $n$ even ($h^6$ for $n\neq 3$ odd) is the generator of automorphisms of infinite order of $\Aut(X)$. Moreover, by Lemma \ref{criterion}, both $h^4$ and $h^6$ are symplectic. 

By the same argument, $h^3$ is the generator of automorphisms of infinite order of $\Aut(X)$ for $n=3$ and is anti-symplectic.

If we assume that $\Aut(X)\cong  \Z_2\ast \Z_2$, then, by Theorem \ref{theorem2},
\begin{equation}\label{tau Example 1}
\tau^*\vert{S_X}=\begin{pmatrix} p&q\\q-np&-p \end{pmatrix},
\end{equation} 
where $p^2-npq+q^2=1$. Moreover, by Remark \ref{a=c}, $\tau^*\vert{S_X}$ or $\sigma^*\vert{S_X}$ is the matrix obtained from $h^k$ by interchanging columns for some $k$. For example, $\tau^*\vert{S_X}:=\tau_k=h^kE_{12}$ or
\begin{equation}\label{}
\tau^*\vert{S_X}=\begin{pmatrix} \alpha_k&\beta_k\\-\beta_k&\alpha_k-n\beta_k \end{pmatrix}\begin{pmatrix} 0&1\\1&0 \end{pmatrix}=\begin{pmatrix} \beta_k&\alpha_k\\ \alpha_k-n\beta_k&-\beta_k \end{pmatrix},
\end{equation} 
where $(\alpha_k, \beta_k)$ is the first row of $h^k$. As in \eqref{power of h}, we have that $\alpha_{k+1}=n\alpha_k-\beta_k, \beta_{k+1}=\alpha_k$.

Since $\tau$ is anti-symplectic, by Lemma \ref{criterion}, $(\tau^*\vert{S_X}+\id)Q_{S_X}^{-1}$ is an integer matrix. Note that if
\begin{equation}\label{tau+id}
(\tau^*\vert{S_X}+\id)Q_{S_X}^{-1}=\frac{1}{2(n^2-4)}\begin{pmatrix} n\alpha_k-2\beta_k-2&-2\alpha_k+n\beta_k+n\\-2\alpha_k+n\beta_k+n&n\alpha_k-2\beta_k-2-(n^2-4)\beta_k \end{pmatrix}
\end{equation} 
is an integer matrix, then $\beta_k$ is even. 

Now if $n$ is odd ($\neq 3$), by induction argument, we can see that $\beta_k$ is even only if $k=3l$ and $\alpha_k$ is even only if $k=3l-1$. So in this case, $-2\alpha_{3l}+n\beta_{3l}+n$ is not divisible by $2$, hence \eqref{tau+id} is not an integer matrix.

On the other hand, if $n$ is even, $\beta_k$ is even only if $k=2l$ and $\alpha_k$ is even only if $k=2l-1$. Hence, $\tau^*\vert{S_X}:=\tau_{2l}^*\vert{S_X}=h^{2l}E_{12}$ for some $l$. Now by multiplying $g=h^4$ or $g^{-1}=h^{-4}$, we may assume that $h^2E_{12}$ or $E_{12}$ is an isometry of $S_X$ of anti-symplectic automorphism since $g$ is symplectic. However, by Lemma \ref{criterion}, this is not possible. 
\end{proof}


\subsection{Example 2}
We consider a K3 surface $X$ whose Picard lattice $S_X$ has the following intersection matrix
\begin{equation}\label{Example 2}
\begin{pmatrix} 2&n\\n&2 \end{pmatrix}.
\end{equation} 
Note that $d:=-\disc(S_X)=n^2-4>0$ and for $n\neq 3$,  $\Aut(X)$ is either $\Z$ or $\Z_2\ast \Z_2$ by Theorem  \ref{Infinite automorphism group}.  In \cite[Example 4]{GLP}, Galluzzi, Lombardo and Peters proved that  $\Aut(X)\cong \Z_2\ast \Z_2$ by finding generators. We can also find the same generators by using our results.

By Theorem \ref{theorem1}, an automorphism $g$ of infinite order in \eqref{Automorphism} satisfies $\gamma=-\beta, \delta=\alpha-n\beta$ and $\alpha^2-n\alpha\beta+\beta^2=1$. Moreover, by Theorem \ref{theorem1}, $g^*\vert S_X$ is a power of $h$, where  
\begin{equation}
h=\begin{pmatrix} n&1\\-1&0 \end{pmatrix}.
\end{equation} 
As the example above, we have that $h^2$ is a symplectic automorphism.

Moreover, by Theorem \ref{theorem2}, we may have that an anti-symplectic involution $\sigma$ which satisfies 
\begin{equation}
\sigma^*\vert S_X=\begin{pmatrix} p&q\\q-np&-p \end{pmatrix},
\end{equation} 
where $p^2-npq+q^2=1$. By Remark \ref{a=c}, $h^2=\sigma^*\vert S_X\circ \tau^*\vert S_X$ and $\sigma^*\vert S_X=hE_{12}, \tau^*\vert S_X=h^{-1}E_{12}$, hence
\begin{equation}
\sigma^*\vert S_X=\begin{pmatrix} 1&n\\0&-1 \end{pmatrix} \text{ and }\tau^*\vert S_X=\begin{pmatrix} -1&0\\n&1 \end{pmatrix}.
\end{equation} 
Moreover, by Lemma \ref{criterion} and Torelli theorem, $\sigma$ and $\tau$ are anti-symplectic involutions. Hence $\Aut(X)\cong \Z_2\ast \Z_2$ with generators $\sigma$ and $\tau$.

\begin{remark}
When $n=4$, $X$ is the complete intersection of bidegree $(1,1)$ and $(2,2)$ hypersurfaces in $\BP^2\times \BP^2$. When $n=5$, $X$ is the complete intersection of bidegree $(1,2)$ and $(2,1)$ hypersurfaces in $\BP^2\times \BP^2$.  
\end{remark}

\end{document}